\documentclass[a4,twoside,10pt]{amsart}

\usepackage{amssymb}
\usepackage[english]{babel}
\usepackage{enumerate}
\usepackage{fontenc}
\usepackage{hyperref} 
\usepackage[latin1]{inputenc}
\usepackage[english]{varioref}
\usepackage[all]{xy}
\usepackage{url}
\usepackage{physics}

\usepackage{todonotes}

\theoremstyle{plain}
\newtheorem{theorem}{Theorem}[section]
\newtheorem{corollary}[theorem]{Corollary}

\newtheorem{proposition}[theorem]{Proposition}

\theoremstyle{definition}
\newtheorem{definition}[theorem]{Definition}
\newtheorem{example}[theorem]{Example}
\newtheorem{construction}[theorem]{Construction}
\newtheorem*{notations}{Notations}

\theoremstyle{remark}
\newtheorem{remark}[theorem]{Remark}
\newtheorem*{acks}{Acknowledgments}

\DeclareMathOperator{\hess}{hess}
\DeclareMathOperator{\Hess}{Hess}

\DeclareMathOperator{\Ann}{Ann}
\DeclareMathOperator{\Hilb}{Hilb}
\DeclareMathOperator{\GCD}{GCD}
\DeclareMathOperator{\Hom}{Hom}

\newcommand{\KK}{\mathbb{K}}
\newcommand{\PP}{\mathbb{P}}

\newcommand{\gen}[1]{\langle#1\rangle}

\begin{document}

\title{CW-complex Nagata Idealizations}

\author{Armando Capasso} \address{Scuola Politecnica e delle Scienze
  di Base, Universit\`a degli Studi di Napoli ``Federico II'', corso Protopisani Nicolangelo 70,
  Napoli (Italy), C.A.P. 80146; \url{armando.capasso@unina.it}}
\author{Pietro De Poi} \address{Dipartimento di Scienze Matematiche,
  Informatiche e Fisiche, Universit\`a degli Studi di Udine, via delle
  Scienze 206, Udine (Italy) C.A.P. 33100;
  \url{pietro.depoi@uniud.it}} \author{Giovanna Ilardi}
\address{Dipartimento di Matematica ed Applicazioni ``R. Caccioppoli'', Universit\`a degli Studi di
  Napoli ``Federico II'', via Cintia
  21, Napoli (Italy) C.A.P. 80126; \url{giovanna.ilardi@unina.it}}

\thanks{P.D.P. \& G.I. are members of INdAM - GNSAGA and P.D.P is
  supported by PRIN2017 ``Advances in Moduli Theory and Birational
  Classification''}

\subjclass[2010]{Primary 13A30, 05E40; Secondary 57Q05, 13D40, 13A02,
  13E10}

\keywords{Lefschetz properties, Artinian Gorenstein Algebra, Nagata
  idealization, CW-complex}

\date{\today}

\begin{abstract}
  We introduce a  construction which allows us to identify the
  elements of the skeletons of a CW-complex $P(m)$ and the monomials
  in $m$ variables. From this, we infer that there is a bijection
  between finite CW-subcomplexes of $P(m)$, which are quotients of
  finite simplicial complexes, and certain  bigraded standard Artinian
  Gorenstein algebras, generalizing previous constructions of Faridi and ourselves.
  
  We apply this to a generalization of Nagata idealization for level
  algebras.  These algebras are standard graded Artinian algebras
  whose Macaulay dual generator is given explicitly as a bigraded
  polynomial of bidegree $(1,d)$. We consider the algebra associated
  to polynomials of the same bidegree $(d_1,d_2)$.
\end{abstract}

\maketitle

\section*{Introduction}
\markboth{Introduction}{Introduction}

Let $X=V(f)\subset\PP^N_{\KK}$ be a hypersurface, where the underlying
field $\KK$ has characteristic $0$; the \emph{Hessian determinant of
  $f$} (which we call \emph{the Hessian of $f$} or
\emph{the Hessian of $X$}) is the \emph{determinant of the Hessian matrix
  of $f$}.

Hypersurface with vanishing Hessian were studied for the first time in
1851 by O. Hesse; he wrote two papers (\cite{OH:1,OH:2}) according to
which these hypersurfaces should be necessarily cones. In  1876
Gordan and Noether (\cite{G:N}) proved that Hesse's claim is true for
$N\leq3$, and it is false for $N\geq4$. They and Franchetta classified
all the counterexamples to Hesse's claim in $\PP^4$
(see~\cite{G:N,F:A,GR,G:R:S}). In 1900, Perazzo classified cubic
hypersurfaces with vanishing Hessian for $N\leq6$ (\cite{P:U}). This
work was studied and generalized in~\cite{GR:RF}, and the problem is
still open in  spaces of  higher dimension.

Hessians of higher degree were introduced in~\cite{M:W} and used to
control the so called \emph{Strong Lefschetz Properties} (for short,
\emph{SLP}). The Lefschetz properties have attracted a great attention
in the last years. The basic papers of the algebraic theory of
Lefschetz properties were the original ones of
Stanley~\cite{SR:1,SR:2,SR:3} and the book of Watanabe and
others~\cite{HMMNWW}.

An algebraic tool that occurs frequently in these papers is the
\emph{Nagata Idealization}: it is a tool to convert any module $M$
over a (commutative) ring (with unit) $R$  to  an ideal of another ring
$R\ltimes M$. The starting point is the isomorphism between the
idealization of an ideal $I=(g_0,\dotsc,g_n)$ of $\KK[u_1,\dotsc,u_m]$
and its level algebra see~\cite[Definition 2.72]{HMMNWW}. In this way,
the new ring is a \emph{Standard Graded Artinian Gorenstein Algebra}
(\emph{SGAG algebra}, for short). An explicit formula for the Macaulay
generator $f$ is:
\begin{equation*}
  f=x_0g_0+\dotsm+x_n g_n\in\KK[x_0,\dotsc,x_n,u_1,\dotsc,u_m]_{(1,d)}.
\end{equation*}
A generalization of this construction is to consider polynomials
of the form:
\begin{equation*}
  f=x_0^dg_0+\dotsm+x_n^dg_n\in\KK[x_0,\dotsc,x_n,u_1,\dotsc,u_m]_{(d,d+1)};
\end{equation*}
these are called \emph{Nagata polynomials of degree $d$}. The
Lefschetz properties for the relevant associated algebras $A$, the
geometry of \emph{Nagata hypersurfaces of degree $d$}, the interaction
between the combinatorics of $f$ and the structure of $A$ were studied
in~\cite{CGIM}, where the $g_i$'s are square free monomials, using a
simplicial complex associated to $f$.

In this paper we use the \emph{CW-complexes}, to study Nagata polynomials of bidegree
$(d_1,d_2)$. We study the Hilbert vector and we give a complete
description of the ideal $I$ for every case, also if the $g_i$'s are
not square free monomials.

The geometry of the Nagata hypersurface is very similar to the
geometry of the hypersurfaces with vanishing Hessian.

More precisely, we introduce a new Construction \ref{con:1} which
allows us to identify each (monic) monomial of degree $d$ in $m$
variables with an element of the $(d-1)$-skeleton of a 
CW-complex that we  call $P(m)$. This CW-complex is constructed
by generalizing the construction introduced in \cite{F:S} which
associates to a (monic) square-free monomial in $m$ variables of
degree $d$ a unique $(d-1)$-cell of the simplex of dimension $m-1$,
and vice versa.   We consider an $h$-power $u_i^h$ as a
product of $h$ linear forms: $\tilde{u}_1\dotsm\tilde{u}_h$; this
corresponds to a $(h-1)$-simplex, and we identify all the
$\delta$-faces, with $\delta<h-1$, of this simplex to just one
$\delta$-face, recursively, starting from $\delta=0$ to $\delta=h-2$:
for $\delta=0$ we identify all the points to one, then if $\delta=1$
we obtain a bouquet of $h$-circles, and we identify all these circles,
and so on. Generalizing this construction to a general monic monomial
and attaching the corresponding CW-complexes along the common
skeletons, we obtain $P(m)$.

The paper is organized as follows: in Section \ref{sec:1} we recall
some generalities about graded Artinian Gorenstein Algebras and
Lefschetz Properties, with their connections with the vanishings of
higher order Hessians. In Section \ref{sec:2} we recall  what the
Nagata idealization is, what we intend for a higher Nagata idealization
and we show its connection with the Lefschetz Properties for
bihomogeneous polynomials. Section \ref{sec:3} is the core of this
article. After recalling 
generalities about bigraded algebras and the topological definitions
that we need, we give the construction of the CW-complex $P(m)$; then,
we apply it to the  Nagata polynomials ( Definition
\ref{def:nag}) in Theorems \ref{thm:1} and \ref{thm:2}, which give
Theorem \ref{thm:1} a precise description of the Artinian Gorenstein
Algebra associated to a Nagata polynomial and Theorem \ref{thm:2} the
generators of the annihilator of the polynomial. We show that from
these theorems  a generalization of the principal results of
\cite{CGIM} follows:  Corollaries \ref{cor:1} and \ref{cor:2}.

We think that the study of the Nagata hypersurfaces can be---among
other things---a  useful tool for the classification of the
hypersurfaces with vanishing Hessian in $\mathbb P^n$.
\begin{acks}
  We thank the anonymous referee for the careful reading and the valuable suggestions.
\end{acks}
\begin{notations}
  In this the paper we fix the following notations and assumptions:
  \begin{itemize}
  \item  $\KK$ is  a field  of characteristic $0$.
  \item $R:=\KK[x_0,\dotsc,x_n]$ will always be the ring of
    polynomials in $n+1$ variables $x_0,\dotsc,x_n$.
  \item $Q:=\KK[X_0,\dotsc,X_n]$ will be the the 
    ring of differential operators of $R$,
    where $\displaystyle X_i=\pdv{x_i}$.
  \item The subscript of a graded $\KK$-algebra will indicate  the part of that degree;  $R_d$ is the
    $\KK$-vector space of the homogeneous polynomials of degree $d$,
    and $Q_{\delta}$ the $\KK$-vector space of the homogeneous
    differential operators of order $\delta$.
  \end{itemize}
\end{notations}


\section{Graded Artinian Gorenstein Algebras and Lefschetz Properties}\label{sec:1}
\markboth{Graded Artinian Gorenstein Algebras and Lefschetz
  Properties}{Graded Artinian Gorenstein Algebras and Lefschetz
  Properties}

\subsection{Graded Artinian Gorenstein Algebras are Poincar\'e
  Algebras}

\begin{definition}
  Let $I$ be a homogeneous ideal of $R$ such that
  $\displaystyle A=R/I=\bigoplus_{i=0}^d A_i$ is a graded Artinian
  $\KK$-algebra, where $A_d\neq 0$.
   The integer $d$ is the \emph{socle degree of $A$}. The algebra $A$ is said \emph{standard}
  if it is generated in degree $1$.  Setting
  $h_i=\dim_{\KK}A_i$, the vector $\Hilb(A)=(1,h_1,\dotsc,h_d)$ is
  called \emph{Hilbert vector of $A$}.  Since $I_1=0$, then $h_1=n+1$
  is called \emph{codimension of $A$}.
\end{definition}
We also recall the following definitions.
\begin{definition}
  A graded Artinian $\KK$-algebra $A=\bigoplus_{i=0}^d A_i$ is a
  \emph{Poincar\'e algebra} if $\cdot\colon A_i\times A_{d-i}\to A_d$
  is a \emph{perfect pairing} for $i\in\{0,\dotsc,d\}$.
\end{definition}

\begin{definition}\label{th2.1}
  A graded Artinian $\KK$-algebra $A$ is \emph{Gorenstein} if (and
  only if) $\dim_{\mathbb{K}}A_d=1$ and it is a Poincar\'e algebra.
\end{definition}

\begin{remark}
  The Hilbert vector of a Poincar\'e algebra $A$ is \emph{symmetric
    with respect to
    $h_{\textstyle\left\lfloor\frac{d}{2}\right\rfloor}$}, that is
  $\Hilb(A)=(1,h_1,h_2,\dotsc,h_2,h_1,1)$.
  \begin{flushright}
    $\Diamond$
  \end{flushright}
\end{remark}

\subsection{Graded Artinian Gorenstein Quotient Algebras of $Q$}

For any $d\geq\delta\geq0$ there exists a natural $\KK$-bilinear map
$B\colon R_d\times Q_{\delta}\to R_{d-\delta}$ defined by
differentiation
\begin{equation*}
  B(f,\alpha)=\alpha(f)
\end{equation*}
\begin{definition}
  Let $I=\gen{f_1,\dotsc, f_\ell}$---where $f_1,\dotsc, f_\ell$ are
  forms in $R$---be a finite dimensional $\KK$-vector subspace of $R$.
  The \emph{annihilator of $I$ in $Q$} is the following homogeneous
  ideal
  \begin{equation*}
    \Ann(I):=\{\alpha\in Q\mid\forall f\in I,\alpha(f)=0\}.
  \end{equation*}
  In particular, if $I$ is generated by a homogeneous element $f$, we
   write  $\Ann(I)=\Ann(f)$.
\end{definition}
Let $A=Q/\Ann(f)$, where $f$ is homogeneous. By construction $A$ is a
standard graded Artinian $\KK$-algebra; moreover $A$ is Gorenstein.
\begin{theorem}[\cite{Mac2}, \S 60ff, \cite{M:W}, Theorem 2.1]\label{th2.2}
  Let $I$ be a homogeneous ideal of $Q$ such that $A=Q/I$ is a
  standard Artinian graded $\KK$-algebra. Then $A$ is Gorenstein if
  and only if there exist $d\geq1$ and $f\in R_d$ such that
  $A\cong Q/\Ann(f)$.
\end{theorem}
\begin{remark}
  Using the notation as above, $A$ is called the \emph{SGAG
    $\KK$-algebra associated to $f$}. The socle degree $d$ of $A$ is
  the degree of $f$ and the codimension is $n+1$, since $I_1=0$.
  \begin{flushright}
    $\Diamond$
  \end{flushright}
\end{remark}

\subsection{Lefschetz Properties and the Hessian Criterion}

Let $\displaystyle A=\bigoplus_{i=0}^d A_i$ be a graded Artinian
$\KK$-algebra.
\begin{definition}
  If there exists an $L\in A_1$ such that:
  \begin{enumerate}
  \item The multiplication map $\cdot L\colon A_i\to A_{i+1}$ is of
    maximal rank for all  $i$, 
    then
    $A$ has the \emph{Weak Lefschetz Property} (\emph{WLP}, for
    short);
  \item The multiplication map $\cdot L^k\colon A_i\to
    A_{i+k}$ is of maximal rank for all $i$ 
    and $k$, 
    then
    $A$ has the \emph{Strong Lefschetz Property} (\emph{SLP}, for
    short);
  \end{enumerate}
\end{definition}
\begin{definition}
  Let $A$ be the SGAG $\KK$-algebra associated to an element $f\in
  R_d$, and let $\mathcal{B}_k=\{\alpha_j\in A_k\mid
  j\in\{1,\dotsc,\sigma_k\}\}$ be an ordered $\KK$-basis of
  $A_k$. The \emph{$k$-th Hessian matrix} of
  $f$ with respect to $\mathcal{B}_k$ is
  \begin{equation*}
    \Hess_f^k=\left(\alpha_i\alpha_j(f)\right)_{i,j=1}^{\sigma_k}.
  \end{equation*}
  The \emph{$k$-th Hessian} of $f$ with respect to $\mathcal{B}_k$ is
  \begin{equation*}
    \hess_f^k=\det\left(\Hess_f^k\right).
  \end{equation*}
\end{definition}

\begin{theorem}[\cite{W:J} Theorem 4]
  An element $L=a_0X_0+\dotsm+a_n X_n\in
  A_1$ is a \emph{strong Lefschetz element} of
  $A$ if and only if
  $\hess_f^k(a_0,\dotsc,a_n)\neq0$ for all $\displaystyle
  k\in\left\{0,\dotsc,\left\lfloor\frac{d}{2}\right\rfloor\right\}$. In
  particular, if for some $k$ one has $\hess_f^k=0$, then
  $A$ does not have  SLP.
\end{theorem}

\section{Higher Order Nagata Idealization}\label{sec:2}
\markboth{Higher Order Nagata Idealization}{Higher Order Nagata
  Idealization}

\subsection{Nagata Idealizations}

\begin{definition}
  Let $A$ be a ring and let $M$ be an $A$-module. The \emph{Nagata
    idealization $A\ltimes M$} of $M$ is the ring with underlying set 
  $A\times M$ and operations defined as follow:
  \begin{equation*}
    (r,m)+(s,n)=(r+s,m+n),\,(r,m)\cdot(s,n)=(rs,sm+r n).
  \end{equation*}
\end{definition}

\subsubsection{Bigraded Artinian Gorenstein Algebras}
Let $\displaystyle A=\bigoplus_{i=0}^d A_i$ be a SGAG $\KK$-algebra,
it is \emph{bigraded} if: 
\begin{equation*}
  A_d=A_{(d_1,d_2)}\cong\KK,\,A_i=\bigoplus_{h=0}^i A_{(i,h-i)}\,\text{\,for\,}\,i\in\{0,\dotsc,d-1\},
\end{equation*}
since $A$ is a Gorenstein ring, and the pair $(d_1,d_2)$ is said the
\emph{socle bidegree of $A$}. In this case we call  $A$  an \emph{SBAG
  algebra}.
\begin{remark}\label{rmk:41}
  By Definition~\ref{th2.1}, $A_i\cong A_{d-i}^{\vee}=\Hom_{\KK}( A_{d-i},\KK)$ and since the
  duality commutes with direct sums, one has
  $A_{(i,j)}\cong A_{(d_1-i,d_2-j)}^{\vee}$ for any pair $(i,j)$.
  \begin{flushright}
    $\Diamond$
  \end{flushright}
\end{remark}
We fix notation as in Theorem~\ref{th3.1}:
  \begin{itemize}
  \item
    $S:=R\otimes_\KK\KK[u_1,\dotsc,u_m]=
    \KK[x_0,\dotsc,x_n,u_1,\dotsc,u_m]$ is the bigraded ring of
    polynomials in $m+n+1$ variables $x_0,\dotsc,x_n,u_1,\dotsc,u_m$;
 \item We have chosen the natural bigrading of $S$: $x_i$ has bidegree $(1,0)$ and $u_j$ has bidegree $(0,1)$;   
  \item Define $S_{(d_1,d_2)}$ to be the $\KK$-vector space of bihomogeneous
    polynomials $f$ of bidegree $(d_1,d_2)$; that is, $f$ can be
    written as $\displaystyle\sum_{i=0}^pa_i b_i$, where
    $a_i\in R_{d_1}=\KK[x_0,\dotsc,x_n]_{d_1}$ and
    $b_i\in\KK[u_1,\dotsc,u_m]_{d_2}$.
  \item
    $T:=Q\otimes_\KK\KK[U_1,\dotsc,U_m]=\KK[X_0,\dotsc,X_n,U_1,\dotsc,U_m]$
    is the (bigraded) ring of differential operators of $S$, where
    $\displaystyle X_i=\pdv{x_i}$ and $\displaystyle U_j=\pdv{u_j}$; $X_i$ has bidegree $(1,0)$ and $U_j$ has bidegree $(0,1)$. 
  \end{itemize}

A homogeneous ideal $I$ of $S$ is a \emph{bihomogeneous ideal} if:
\begin{equation*}
  I=\bigoplus_{i,j=0}^{\infty}I_{(i,j)},\text{\,where\,}\,\forall i,j\in\mathbb{N}_{\geq0},\,I_{(i,j)}=I\cap S_{(i,j)}.
\end{equation*}
Let $f\in S_{(d_1,d_2)}$, then $I=\Ann(f)$ is a bihomogeneous ideal and 
using Theorem~\ref{th2.2}, $A=T/(\Ann(f))$ is a SBAG $\KK$-algebra of
socle bidegree $(d_1,d_2)$ (and codimension $m+n+1$).
\begin{remark}
  Using the above notations, one has:
  \begin{equation*}
    \forall i>d_1,j>d_2,\,I_{(i,j)}=T_{(i,j)}.
  \end{equation*}
  Indeed, for all $\alpha\in T_{(i,j)}$ with $i>d_1,j>d_2$,
  $\alpha(f)=0$; as a consequence:
  \begin{equation*}
    \forall k\in\{0,\dotsc,d_1+d_2\},\,A_k=\bigoplus_{\substack{0\leq i\leq d_1\\
        0\leq j\leq d_2\\
        i+j=k}}A_{(i,j)}.
  \end{equation*}
  Moreover, the evaluation map
  $\alpha\in T_{(i,j)}\mapsto \alpha(f)\in A_{(d_1-i,d_2-j)}$ provides
  the following short exact sequence:
  \begin{equation}\label{SES}
    \xymatrix{
      0\ar[r] & I_{(i,j)}\ar[r] & T_{(i,j)}\ar[r] & A_{(d_1-i,d_2-j)}\ar[r] & 0.
    }
  \end{equation}
  \begin{flushright}
    $\Diamond$
  \end{flushright}
\end{remark}

The following theorem, which links Nagata idealizations with bihomogeneous polynomials, holds.
\begin{theorem}[\cite{HMMNWW}, Theorem 2.77]\label{th3.1}
  Let $S^{\prime}:=\KK[u_1,\dotsc,u_m]$ and
  $S:=R\otimes_\KK S^{\prime}$ be rings of polynomials, let
  $T^{\prime}=\KK[U_1,\dotsc,U_m]$ and $T:=Q\otimes_\KK T^{\prime}$ be
  the associated ring of differential operators, where
  $\displaystyle X_i=\pdv{x_i}$ and $\displaystyle U_j=\pdv{u_j}$. Let
  $g_0,\dotsc,g_n$ be homogeneous elements of $S^{\prime}$ of degree
  $d$, let $I$ be the $T^{\prime}$-submodule of $S^{\prime}$ generated
  by $\{\partial(g_i)\in R\mid\partial\in T,i\in\{0,\dotsc,n\}\}$ and
  let $A^{\prime}:=T^{\prime}/\Ann(I)$.  Define
  $f=x_0g_0+\dotsm+x_n g_n\in R$, it is a bihomogeneous polynomial of
  bidegree $(1,d)$, and let $A:=T/\Ann(f)$. Considering $I$ as an
  $A^{\prime}$-module,  $A^{\prime}\ltimes I\cong A$.
\end{theorem}

\subsection{Lefschetz Properties for Higher Nagata Idealizations}

\begin{definition}\label{def:nag}
  A bihomogeneous polynomial
  \begin{equation*}
    f=\sum_{i=0}^n x_i^{d_1}g_i\in S_{(d_1,d_2)}
  \end{equation*}
  is called a \emph{CW-Nagata polynomial of degree $d_1\geq1$} if
  $g_i\in \KK[u_1,\dotsc,u_m]$, $i=0,\dotsc,n$, are linearly
  independent monomials of degree $d_2\geq 2$.
\end{definition}
\begin{remark}
  One needs $\displaystyle n\leq\binom{m+d_2-1}{d_2}$ otherwise the
  $g_i$  cannot be linearly independent.

  From now on, we assume that $n$ satisfies this condition.
  \begin{flushright}
    $\Diamond$
  \end{flushright}
\end{remark}

We will need  the following propositions.
\begin{proposition}[\cite{GR} Proposition 2.5]
  Let  $n+1\geq m\geq 2,d_2>d_1\geq1$ and
  $\displaystyle s>\binom{m+d_1-1}{d_1}$; for any
  $j\in\{1,\dotsc,s\}$, let $f_j\in S_{(d_1,0)},\,g_j\in
  S_{(0,d_2)}$. Then the form $f=f_1g_1+\dotsm+f_s g_s$ of degree
  $d_1+d_2$ satisfies
  \begin{equation*}
    \hess_f^{d_1}=0;
  \end{equation*}
  that is, $A=T/\Ann(f)$ does not have the SLP condition.
\end{proposition}
\begin{proposition}[\cite{CGIM} Proposition 2.7]
  Let  $n+1\geq m\geq2,d_1\geq d_2$.  Then
  $\displaystyle L=\sum_{i=0}^n X_i$ is a \emph{Weak Lefschetz
    Element}; that is, $A=T/\Ann(f)$ has the WLP condition.
\end{proposition}

\section{CW-complex Nagata Idealization of Bidegree $(d_1,d_2)$}\label{sec:3}
\markboth{CW-complex Nagata Idealization of Bidegree
  $(d_1,d_2)$}{CW-complex Nagata Idealization of bidegree $(d_1,d_2)$}

Let  $S$ and $T$  be as in the previous subsection.

\begin{definition}
  A bihomogeneous CW-Nagata polynomial
  \begin{equation*}
    f=\sum_{i=0}^n x_i^{d_1}g_i\in S_{(d_1,d_2)}
  \end{equation*}
  is called a \emph{simplicial Nagata polynomial of degree $d_1$} if
  the monomials $g_i$ are square free.
\end{definition}
\begin{remark}
  One needs $\displaystyle n\leq\binom{m}{d_2}$ otherwise the $g_i$
  cannot be square free.
  \begin{flushright}
    $\Diamond$
  \end{flushright}
\end{remark}

\subsection{CW-complexes and bihomogeneous polynomials}

\subsubsection{Abstract finite simplicial complexes}
\begin{definition}
  Let $V=\{u_1,\dotsc,u_m\}$ be a finite set. An  \emph{abstract
    simplicial complex $\Delta$ with vertex set $V$} is
  a subset of $2^V$ such that
  \begin{enumerate}
  \item $\forall u\in V\Rightarrow\{u\}\in\Delta$,
  \item
    $\forall
    \sigma\in\Delta,\tau\subsetneqq\sigma,\tau\neq\emptyset\Rightarrow\tau\in\Delta$.
  \end{enumerate}
\end{definition}
The elements $\sigma$ of $\Delta$ are called \emph{faces} or
\emph{simplices}; a face with $q+1$ vertices is called \emph{$q$-face}
or \emph{face of dimension $q$} and one writes $\dim\sigma=q$; the
maximal faces (with respect to the inclusion) are called
\emph{facets}; if all facets have the same dimension $d\geq1$ then one
says that $\Delta$ is \emph{of pure dimension $d$}.  The set
$\Delta^k$ of faces of dimension at most $k$ is called
\emph{$k$-skeleton of $\Delta$}.  $2^V$ is called \emph{simplex} (of
dimension $m-1$).
\begin{remark}\label{rem1}
  \
  \begin{enumerate}
  \item (cfr.~\cite[Remark 3.4]{CGIM}) There is a natural bijection,
     introduced  in \cite{F:S}, between the
    square free monomials, of degree $d$, in the variables
    $u_1,\dotsc,u_m$ and the $(d-1)$-faces of the simplex $2^V$, with
    vertex set $V=\{u_1,\dotsc,u_m\}$.  In fact, a square free
    monomial $g=u_{i_1}\dotsm u_{i_d}$ corresponds to the subset
    $\{u_{i_1},\dotsc,u_{i_d}\}$ of $2^V$.  Vice versa, to any subset
    $F$ of $V$ with $d$ elements one associates the free square
    monomial $\displaystyle m_F=\prod_{u_i\in F}u_i$ of degree $d$.
  \item Let
    $\displaystyle f=\sum_{i=0}^n x_i^{d_1}g_i\in S_{(d_1,d_2)}$ be a
    simplicial Nagata polynomial; by hypothesis there is bijection
    between the monomials $g_i$  and the indeterminates $x_i$.
    From  this, we can associate to $f$ a simplicial complex
    $\Delta_f$ with vertices $u_1,\dotsc,u_m$ where the facet which
    corresponds to $g_i$ is identified with $x_i^{d_1}$.
  \end{enumerate}
  \begin{flushright}
    $\Diamond$
  \end{flushright}
\end{remark}

\subsubsection{CW-complexes}

For the topological background, we refer to~\cite{H:AE}. We start by fixing some notations. 
\begin{definition}\label{cell}
  Let $k\in\mathbb{N}_{\geq1}$. A topological space $e^k$ homeomorphic
  to the open (unitary) ball 
$\{(x_1,\dotsc,x_k)\in\mathbb{R}^k\mid x_1^2+\dotsm+x_k^2<1\} $
of dimension $k$ (with the  natural topology induced by
$\mathbb{R}^{k+1}$) is called a \emph{$k$-cell}. Its boundary, i.e. the
\emph{$(k-1)$-dimensional sphere} will be denoted by
$\mathbb{S}^{k-1}=\{(x_1,\dotsc,x_k)\in\mathbb{R}^k\mid x_1^2+\dotsm+x_k^2=1\}$ and
its closure, \ i.e. the closed (unitary) $k$-dimensional disk will be denoted by
$\mathbb{D}^k:=\{(x_1,\dotsc,x_k)\in\mathbb{R}^k\mid x_1^2+\dotsm+x_k^2\leq1\}$.  
\end{definition}

We recall the following 
\begin{definition}
  A \emph{CW-complex} is a topological space $X$ constructed in the
  following way:
  \begin{enumerate}
  \item There exists a fixed and discrete set of points
    $X^0\subset X$, whose elements are called \emph{$0$-cells};
  \item Inductively, the \emph{$k$-skeleton $X^k$ of $X$} is
    constructed from $X^{k-1}$ by attaching $k$-cells $e_{\alpha}^k$
    (with index set $A_k$) via continuous maps
    $\varphi^k_{\alpha}\colon\mathbb{S}^{k-1}_{\alpha}\to X^{k-1}$
    (the \emph{attaching maps}).  This means that $X^k$ is a quotient
    of
    $\displaystyle Y^k=X^{k-1}\bigcup_{\alpha\in
      A_k}\mathbb{D}^k_{\alpha}$ under the identification
    $x\sim\varphi_{\alpha}(x)$ for
    $x\in\partial\mathbb{D}^k_{\alpha}$; the \emph{elements of the
      $k$-skeleton} are the (closure of the) attached $k$ cells;
  \item $\displaystyle X=\bigcup_{k\in\mathbb{N}_{\geq0}}X^k$ and a
    subset $C$ of $X$ is closed if and only if $C\cap X^k$ is closed
    for any $k$ (\emph{closed weak topology}).
  \end{enumerate}
\end{definition}

\begin{definition}
  A subset $Z$ of a CW-complex $X$ is a \emph{CW-subcomplex} if it is
  the union of cells of $X$, such that the closure of each cell is in
  $Z$.
\end{definition}

\begin{definition}
  A CW-complex is \emph{finite} if it consists of a finite number
  of cells.
\end{definition}

We will be interested mainly in finite CW-complexes.

\begin{example}[Geometric realization of an abstract simplicial
  complex]\label{ex:1} 
  It is an obvious fact that to any simplicial complex $\Delta$ one can associate a
  finite CW-complex $\widetilde{\Delta}$  via the \emph{geometric
    realization} of $\Delta$ as a
  simplicial complex (as a topological space) $\widetilde{\Delta}$. 

  \begin{flushright}
    $\triangle$
  \end{flushright}
\end{example}
 
  In what follows we will always identify abstract simplicial
  complexes with their corresponding simplicial complexes.
\begin{construction}\label{con:1}
  In  Remark~\ref{rem1}, we saw that to any
  degree $d$ square-free monomial
  $u_{i_1}\dotsm u_{i_d}\in\KK[u_1,\dotsc,u_m]_d$ one can associate
  the $(d-1)$-face $\{u_{i_1},\dotsc,u_{i_d}\}$ of the abstract
  $(m-1)$-dimensional simplex $\Delta(m):=2^{\{u_1,\dotsc,u_m\}}$, and
  vice versa: if we call
  \begin{align*}
    \rho_d&:=\{f\in \KK[u_1,\dotsc,u_m]_d\mid \textup{$f\neq 0$ is a square-free monic monomial}\}\\
    D(m)_d&:=\Delta(m)^d\setminus\Delta(m)^{d-1},
  \end{align*}
  we have a bijection
  \begin{align*}
    \sigma_d\colon\rho_d & \to D(m)_d\\
    u_{i_1}\dotsm u_{i_d}&\mapsto\left\{u_{i_1},\dotsc,u_{i_d}\right\}.
  \end{align*}
  Alternatively, we can associate to $u_{i_1}\dotsm u_{i_d}$ the
  element of the $(d-1)$-skeleton
  $\overline{\left\{u_{i_1},\dotsc,u_{i_d}\right\}}\in\widetilde{\Delta(m)}^{d-1}$,
  so  we have a bijection
  \begin{align*}
    \sigma_d\colon\rho_d &\to\widetilde{\Delta(m)}^{d-1}\\
    u_{i_1}\dotsm u_{i_d}&\mapsto\overline{\left\{u_{i_1},\dotsc,u_{i_d}\right\}}
  \end{align*}
  between the square-free monomials and the $(d-1)$-faces of the
  (topological) simplex $\widetilde{\Delta(m)}$.

  Using CW-complexes, we will  extend this construction to the
  \emph{non-square-free monic monomials}.  We  proceed as
  follows. Let $g:=u_1^{j_1}\dotsm u_m^{j_m}$ be a generic degree
  $d:=j_1+\dotsb+j_m$ monomial.  Consider the following finite set:
  $W:=\left\{u^1_1,\dotsc,u^{j_1}_1,\dotsc,u^1_m,\dotsc,u^{j_m}_m\right\}$,
  and if $\Delta(d):=2^W$ is the abstract associated (finite) simplex,
  we  consider the corresponding
  (topological) simplex (which is a CW-complex)
  $\widetilde{\Delta(d)}$.

  If  $j_k\leq1$ we do nothing, while if $j_k\ge 2$, we
  recursively identify, for $\ell$ varying from $0$ to $j_k-2$, the
  $\ell$-faces of the subsimplex
  $\widetilde{2^{\left\{u^1_k,\dotsc,u^{j_k}_k\right\}}}\subset
  \widetilde{\Delta(d)}$: start with $\ell=0$, and we identify all the
  $j_k$ points to one point---call it $u_k$. Then, for $\ell=1$, we
  obtain a bouquet of $\binom{j_k+1}{2}$ circles, and we identify them
  in just one circle $\mathbb S^1$ passing through $u_k$, and so on,
  up to the facets of
  $\widetilde{2^{\left\{u^1_k,\dotsc,u^{j_k}_k\right\}}}$, i.e. its
  $j_k+1$ $(j_k-1)$-faces, which, by the construction, have all their
  boundary in common, and we identify all of them.

  Make all these identifications for all $j_1,\dotsc ,j_m$; in this
  way, we obtain a finite CW-complex $X=X_g$ of dimension $d-1$, with
  $0$-skeleton $X^0=\{u_i\mid j_i\neq0\}\subset \{u_1,\dotsc,u_m\}$,
  obtained from the $(d-1)$-dimensional simplex
  $\widetilde{\Delta(d)}$, with the above identification.

  In this way, we obtain a finite CW-complex
  $X=X_g$ 
  of dimension $d-1$, with $0$-skeleton
  $X^0=\{u_i\mid j_i\neq0\}\subset \{u_1,\dotsc,u_m\}$, obtained from
  the $(d-1)$-dimensional simplex $\widetilde{\Delta(d)}$, with the
  above identification. Under this identification  each closure of a $(j_k-1)$-cell
  $\overline{\left\{u_k^1,\dotsc,u_k^{j_k}\right\}}$ becomes  a
    point if $j_k=1$, a circle $\mathbb{S}^{1}$ if $j_k=2$, a
  topological space with fundamental group $\mathbb Z_3$ if $j_k=2$
  (i.e. it is not a topological surface), etc. We will denote these
  spaces in what follows by $\epsilon^{j_k-1}_k$,
  i.e. $\epsilon^{j_k-1}_k$ corresponds to $u_k^{j_k}$, and vice
  versa:
    \begin{proposition}
    Every power in $u_1^{j_1}\dotsm u_m^{j_m}$ (up to a permutation of
    the variables) corresponds to a $\epsilon^{j_k-1}_k$, and vice
    versa.
  \end{proposition}
    We can see $X_g$ as a $(d-1)$-dimensional 
  \emph{join} between these spaces $\epsilon^{j_k-1}_k$ and the span
  of the $0$-skeleton $X^0$ i.e. the simplex
  $S_X\subset\widetilde{\Delta(m)}$ associated to it;
  $S_X\cong \widetilde{\Delta(\ell)}$, where $\ell=\# X^0\le m$.
  \begin{remark}
    This last observation suggests we  consider an alternative
    construction:  recall that the cellular decomposition of the
    real projective space is obtained attaching a single cell at each
    passage; indeed, $\PP_{\mathbb R}^n$ is obtained from
    $\PP_{\mathbb R}^{n-1}$ by attaching one $n$-cell with the
    quotient projection
    $\varphi^{n-1}\colon \mathbb S^{n-1}\to \PP_{\mathbb R}^{n-1}$ as
    the attaching map.

    Then, to each power
    $u_k^{j_k}$ we associate a real projective space of dimension
    $j_k-1$
    $\PP_k^{j_k-1}$ and immersions
    $i_{k-1}\colon\PP_k^{j_k-1}\hookrightarrow
    \PP_k^{j_k}$; so $\PP_k^0=u_k\in \PP_k^{j_k-1}$.

    Finally, to $g=u_1^{j_1}\dotsm u_m^{j_m}$ we associate the join
    between the $\PP_k^{j_k-1}$ and the $S_X$ defined above; if we
    call this join by $X_g$, we can proceed in an equivalent way, by
    changing $\epsilon^{j_k-1}_k$ with $\PP_k^{j_k-1}$.
    \begin{flushright}
      $\Diamond$
    \end{flushright}
  \end{remark}
  It is  clear how to glue two of these finite CW-complexes---say
  $X=X_{u_1^{j_1} \dotsm u_m^{j_m}}$ and
  $Y=Y_{u_1^{k_1} \dotsm u_m^{k_m}}$, of degree $d=j_1+\dotsb+j_m$ and
  $d^{\prime}=k_1+\dotsb+k_m$---along $\widetilde{\Delta(m)}$: we
  simply attach $X$ and $Y$ via the inclusion maps
  $S_X\subset \widetilde{\Delta(m)}$ and
  $S_Y\subset \widetilde{\Delta(m)}$, where $S_X$ and $S_Y$ are the
  simplexes associated to, respectively, $X$ and $Y$.

  Finally, taking all these finite CW-complexes together, we obtain a
  CW-complex $P$ in the following way:
  \begin{align*}
    C&:=\bigsqcup_{ u_1^{j_1} \dotsm u_m^{j_m}\in \KK[u_1,\dotsc,u_m]}X_{u_1^{j_1} \dotsm u_m^{j_m}} & P(m):={C} /{\sim}
  \end{align*}
  where $\sim$ is the equivalence relation induced by the above
  gluing.
  
  \begin{proposition}
    There is bijection between the monomials of degree $d$ in
    $\KK[u_1,\dotsc,u_m]$ and the elements of the $(d-1)$-skeleton of
    $P(m)$.
  \end{proposition}
  In other words, if we define
  \begin{equation*}
    \rho^{\prime}_d:=\{f\in \KK[u_1,\dotsc,u_m]_d\mid \textup{$f\neq 0$ is a monic monomial}\}
  \end{equation*}
  we have a bijection, using the above notation
  \begin{align*}
    \sigma^{\prime}_d\colon\rho^{\prime}_d & \to P(m)_{d-1}\\
    u_1^{j_1} \dotsm u_m^{j_m}&\mapsto X_{u_1^{j_1} \dotsm u_m^{j_m}}.
  \end{align*}
  \begin{proposition} \label{prop:xx}
    $X_{u_1^{j_1} \dotsm u_m^{j_m}}\subset X_{u_1^{k_1} \dotsm
      u_m^{k_m}}$ if and only if $u_1^{j_1} \dotsm u_m^{j_m}$ divides
    $u_1^{k_1} \dotsm u_m^{k_m}$.
  \end{proposition}
  Let $\displaystyle f=\sum_{i=0}^n x_i^{d_1}g_i\in S_{(d_1,d_2)}$ be
  a CW-Nagata polynomial; by hypothesis there is bijection between the
  monomials $g_i$ and the indeterminates $x_i$. From  this, we
  can associate to $f$ a finite $(d_2-1)$-dimensional, CW-subcomplex
  of $P(m)$, $\Delta_f$ where the $(d_2-1)$-skeleton is given by the
  $X_{g_i}$'s glued together with the above procedure. Each $X_{g_i}$
  can be identified with $x_i^{d_1}$ as before.
\end{construction}
The previous construction generalizes the analogous one given
in~\cite{CGIM}.

\subsection{The Hilbert Function of SBAG Algebras}

The first main result of this paper is the following general theorem.
\begin{remark}\label{rmk:1}
  In order to state it, we observe that the canonical bases of
  \begin{displaymath}
    S_{(d_1,d_2)}=\KK[x_0,\dotsc,x_n]_{d_1}\otimes\KK[u_1,\dotsc,u_m]_{d_2}
  \end{displaymath}
  and
  \begin{displaymath}
    T_{(d_1,d_2)}=\KK[X_0,\dotsc,X_n]_{d_1}\otimes\KK[U_1,\dotsc,U_m]_{d_2}
  \end{displaymath}
  given by monomials are dual bases each other, i.e.
  \begin{equation*}
    X_0^{k_0}\dotsm X_n^{k_n}U_1^{\ell_1}\dotsm U_m^{\ell_m} (x_0^{i_0}\dotsm x_n^{i_n}u_1^{j_1}\dotsm u_m^{j_m})=\delta^{i_0,\dotsc,i_n,j_1,\dotsc,j_m}_{k_0,\dotsc,k_n,\ell_1,\dotsc,\ell_m}
  \end{equation*}
  where $i_0+\dotsm+i_n=k_0+\dotsm+k_n=d_1$,
  $j_1+\dotsm+j_m=\ell_1+\dotsm+\ell_m=d_2$ and
  $\delta^{i_0,\dotsc,i_n,j_1,\dotsc,j_m}_{k_0,\dotsc,k_n,\ell_1,\dotsc,\ell_m}$
  is the Kronecker delta.

  This simple observation allows us to identify---given a CW-Nagata
  polynomial
  $\displaystyle f=\sum_{r=0}^n x_r^{d_1}g_r\in S_{(d_1,d_2)}$---
  the \emph{dual differential operator $G_r$} of the monomial
  $g_r$---i.e. the monomial $G_r\in \KK[U_1,\dotsc,U_m]_{d_2}$ such
  that $G_r(g_r)=1$ and $G_r(g)=0$ for any other monomial
  $g\in \KK[u_1,\dotsc,u_m]_{d_2}$---with the same element of the
  $(d_2-1)$-skeleton of $
  {\Delta}_f$ associated to $g_r$. In other words, we associate to
  $g_r=u_{1}^{j_1}\dotsm u_{m}^{j_m}$ and to
  $G_r=U_{1}^{j_1}\dotsm U_{m}^{j_m}$ the CW-subcomplex of
  $\Delta_f\subset P(m)$, $X_{u_{1}^{j_1}\dotsm u_{m}^{j_m}}$.
\end{remark}

\begin{theorem}\label{thm:1}
  Let $\displaystyle f=\sum_{r=0}^nx^{d_1}_rg_r\in R_{(d_1,d_2)}$,
  with $g_r=u_{1}^{j_1}\dotsm u_{m}^{j_m}$, be a 
  CW-Nagata polynomial of (positive) degree $d_1$, where
  $\displaystyle n\leq\binom{m}{d_2}$, let $\Delta_f$ be the
  CW-complex associated to $f$ and let $A=Q/\Ann(f)$. Then
  \begin{equation*}
    A=\bigoplus_{h=0}^{d=d_1+d_2}A_h
  \end{equation*}
  where
  \begin{equation*}
    A_h=A_{(h,0)}\oplus\dotsb\oplus A_{(p,q)}\oplus\dotsb\oplus A_{(0,h)},\,p\leq d_1,\,q\leq d_2,\,A_d=A_{(d_1,d_2)} 
  \end{equation*}
  and moreover, $\forall j\in\{0,1,\dotsc,d_2\}$,
  \begin{equation*}
    \dim A_{(i,j)}=a_{i,j}=\begin{cases}
      f_j & i=0\\
      \displaystyle\sum_{r=0}^n f_{j,r} & i\in\{1,\dotsc,d_1-1\},\\  
      f_{d_2-j} & i=d_1
    \end{cases}
  \end{equation*}
  where:
  \begin{itemize}
  \item $f_j$ is the number of the elements of the $(j-1)$-skeleton of
    the CW-complex $\Delta_f$ (with the convention that $f_0=1$);
  \item $f_{j,r}$ is the number of the elements of the
    $(j-1)$-skeleton of the CW complex $X_{G_r}$ (with the convention
    that $f_{0,r}=1$, so that $\dim A_{(i,0)}=n+1$).
  \end{itemize}
  More precisely, a basis for $A_{(i,j)}$,
  $\forall j\in\{0,1,\dotsc,d_2\}$, is given by
  \begin{enumerate}
  \item\label{1:1} If $i=0$, $\{\Omega_1,\dotsc,\Omega_{f_j}\}$, where
    any $\Omega_s:=U_{1}^{s_1}\dotsm U_m^{s_{m}}$, with
    $s_1+\dotsm+s_m=j$, is associated to the element
    $X_{u_{1}^{s_1}\dotsm u_m^{s_{m}}}$ of the $(j-1)$-skeleton of
    $\Delta_f$;
  \item\label{1:2} If $i=1,\dotsc, d_1-1$, $\left\{\Omega_s^{i,s_1,\dotsc,s_m}\right\}_{\substack{s\in\{0,\dotsc,n\}\\
        s_k\le,r_k,k=1,\dotsc,m\\ \sum_k s_k=j}}$ where
    $\Omega_s^{i,s_1,\dotsc,s_m}:=X_s^i\cdot U_{1}^{s_1}\dotsm
    U_m^{s_{m}}$ is associated to the element
    $X_{u_{1}^{s_1}\dotsm u_m^{s_{m}}}$ of the $(j-1)$-skeleton of
    $X_{g_s}$;
  \item\label{1:3} If
    $i=d_1,\,\left\{X_0^{d_1}\Omega_1(f),\dotsc
      X_n^{d_1}\Omega_{f_{d_2}-j}(f)\right\}$, where
    $\left\{\Omega_1,\dotsc\Omega_{f_{d_2-j}}\right\}$ is the basis
    for $A_{(0,d_2-j)}$ of case \eqref{1:1}.
  \end{enumerate}
  In the cases \eqref{1:1} and \eqref{1:2} the basis are given by
  monomials, in the case \eqref{1:3}, in general, not.
\end{theorem}

\begin{proof}
  We divide the proof into computing the dimension of $A_{(i,j)}$ and
  find a basis for it, as $i$ varies:

  \begin{itemize}

  \item[$i=0$:]  $A_{(0,0)}\cong\KK$.

    Then, by definition, if $j\in\{1,\dotsc,d_2\}$, $A_{(0,j)}$ is
    generated by the  (canonical images of the) monomials
    $\Omega_s\in Q_j=\KK[U_1,\dotsc,U_m]_j\cong Q_{(0,j)}$ that do not
    annihilate $f$. This means that, if we write
    \begin{align*}
      \Omega_s&=U_{1}^{s_1}\dotsm U_m^{s_{m}} & s_1+\dotsm+s_m&=j,
    \end{align*}
    there exists an $r_s\in\{0,\dotsc,n\}$ such that
    $g_{r_s}=u_{1}^{s_1}\dotsm u_m^{s_{m}}g^{\prime}_{r_s}$, where
    $g^{\prime}_{r_s}\in R_{d_2-j}$ is a (nonzero) monomial; this
    means that $X_{u_{1}^{s_1}\dotsm u_m^{s_{m}}}$ is an element of
    the $(j-1)$-skeleton of the CW-complex $\Delta_f$ by Proposition
    \ref{prop:xx}.

    We need  to prove that these monomials are linearly independent
    over $\KK$: let $\{\Omega_1,\dotsc,\Omega_{f_j}\}$ be a system of
    monomials of $Q_{(0,j)}$, where any
    $\Omega_s=U_{1}^{s_1}\dotsm U_m^{s_{m}}$ with $s_1+\dotsm+s_m=j$,
    is associated to an element of the $(j-1)$-skeleton of the
    CW-complex $\Delta_f$; take a linear combination of them and apply
    it to $f$:
    \begin{equation*}
      0=\sum_{s=1}^{f_j}c_s\Omega_s(f)=\sum_{s=1}^{f_j}c_s\sum_{r=0}^nx_r^{d_1}\Omega_s(g_r)=\sum_{r=0}^nx_r^{d_1}\sum_{s=1}^{f_j}c_s\Omega_s(g_r). 
    \end{equation*}
    By the linear independence of the $x_r^{d_1}$'s
    \begin{align}\label{eq:omega}
      \sum_{s=1}^{f_j}c_s\Omega_s(g_r)&=0,&\forall r&\in\{0,\dotsc,n\}.
    \end{align}
    By hypothesis, for any index $s$ there exists an
    $r_s\in\{0,\dotsc,n\}$ such that
    $\Omega_{s}(g_{r_s})=g^{\prime}_{r_s}\in R_{d_2-j}\setminus\{0\}$,
    then for any index $s$ one has $c_s=0$, since the linear
    combinations in \eqref{eq:omega} are formed by linearly
    independent monomials ($g_r$ is fixed in each linear
    combination!). In other words, $\dim A_{(0,j)}=f_j$.

  \item[$0<i<d_1$:]  Observe that $X_a X_b(f)=0$ if
    $a\neq b$.  Therefore $A_{(i,j)}$ is generated by the only
    (canonical images of) the monomials
    $\Omega_s^{i,s_1,\dotsc,s_m}:=X_s^iU_{1}^{s_1}\dotsm
    U_{m}^{s_m}\in Q_{(i,j)}$, with $s_1+\dotsm+ s_m=j$, that do not
    annihilate $f$. In particular, a basis for $A_{(i,0)}$ is given by
    $X_0^i,\dotsc,X_n^i$ and we can suppose from now on that
    $j>0$. Since
    \begin{equation*}
      \Omega_s^{i,s_1,\dotsc,s_m}(f)=x_s^{d_1-i}\left(U_{1}^{s_1}\dotsm U_{m}^{s_m}\right)(g_s),
    \end{equation*}
    in order to obtain that this is not zero, we must have that
    $g_s=u_{1}^{s_1}\dotsm u_{m}^{s_m}g^{\prime}_s$, where
    $g^{\prime}_{s}\in R_{d_2-j}$ is a nonzero monomial.  This means
    $X_{u_{1}^{s_1}\dotsm u_{m}^{s_m}}\subset X_{g_s}$ by Proposition
    \ref{prop:xx}.

    As above, we can prove that these monomials are linearly
    independent over $\KK$:
    let $$\left\{\Omega_s^{i,s_1,\dotsc,s_m}\right\}_{\substack{s\in\{0,\dotsc,n\}\\ s_k\le,r_k, k=1,\dotsc,m\\
        \sum_k s_k=j}}$$ be a system of monomials of $Q_{(i,j)}$, where
    any
    $\Omega_s^{i,s_1,\dotsc,s_m}=X_s^i\cdot U_{1}^{s_1}\dotsm
    U_{m}^{s_m}$ is associated to the element
    $X_{u_{1}^{s_1}\dotsm u_{m}^{s_m}}$ of the $(j-1)$ skeleton of
    $X_{g_s}\subset\Delta_f$, i.e.
    $X_{u_{1}^{s_1}\dotsm u_{m}^{s_m}}\subset X_{g_s}\subset\Delta_f$
    by Proposition \ref{prop:xx}.

    Take a linear combination of them and apply it to $f$:
    \begin{equation}\label{eq:tanti}
      0=\sum_{\substack{s\in\{0,\dotsc,n\}\\ s_k\le,r_k, k=1,\dotsc,m\\
          \sum_k s_k=j}}c_s^{i,s_1,\dotsc,s_m}\Omega_s^{i,s_1,\dotsc,s_m}(f)
      =\sum_{s=0}^{n}x^{d_1-i}\sum_{\substack{s_k\le,r_k, k=1,\dotsc,m\\
          \sum_k s_k=j}}c_s^{i,s_1,\dotsc,s_m}g_s^{i,s_1,\dotsc,s_m}
    \end{equation}
    where $g_s^{i,s_1,\dotsc,s_m}\in R_{d_2-j}$ is the nonzero
    monomial such that
    $g_s=u_{1}^{s_1}\dotsm u_{m}^{s_m}g_s^{i,s_1,\dotsc,s_m}$.  From
    \eqref{eq:tanti} we deduce, as in the preceding case, that
    \begin{align}\label{eq:bis}
      \sum_{\substack{s_k\le,r_k, k=1,\dotsc,m\\
      \sum_k s_k=j}}c_s^{i,s_1,\dotsc,s_m}g_s^{i,s_1,\dotsc,s_m}&=0 & s&=0,\dotsc, n;
    \end{align}
    as before, given one choice of $s_1,\dotsc,s_m$ there exists an
    $s\in\{0,\dotsc,n\}$ such $g_s^{i,s_1,\dotsc,s_m}(f)$ is a nonzero
    monomial, and the (nonzero) $g_s^{i,s_1,\dotsc,s_m}$'s in
    \eqref{eq:bis} are linearly independent since are obtained by a
    fixed $g_s$.

  \item[$i=d_1$:] By duality, see Remark \ref{rmk:41},
    $A_{(d_1,j)}\cong A^{\vee}_{(0,d_2-j)}$ so
    $\dim A_{(d_1,j)}=f_{d_2-j}$.  To find a basis for $A_{(d_1,j)}$,
    we consider the exact sequence \eqref{SES} given by evaluation at 
    $f$, which in this case reads
    \begin{equation}\label{eq:seq}
      0\to I_{(0,d_2-j)}\to Q_{(0,d_2-j)}\to A_{(d_1,j)}\to 0,
    \end{equation}
    then a basis for $A_{(d_1,j)}$ is obtained in the following way:
    if $\{\Omega_1,\dotsc\Omega_{f_{d_2-j}}\}$ is the basis for
    $A_{(0,d_2-j)}\cong Q_{(0,d_2-j)}/I_{(0,d_2-j)}$ of the case
    $i=0$, then a basis for $A_{(d_1,j)}$ is
    $\left\{X_0^{d_1}\Omega_1,\dotsc,X_n^{d_1}\Omega_{f_{d_2}-j}(f)\right\}$.
  \end{itemize}
\end{proof}

As a corollary of Theorem \ref{thm:1} we see that we can deduce the
general case of the simplicial Nagata polynomial, which is a
slight improvement of the first part of \cite[Theorem 3.5]{CGIM}.

\begin{corollary}\label{cor:1}
  Let $\displaystyle f=\sum_{r=0}^nx^{d_1}_rg_r\in R_{(d_1,d_2)}$,
  with $g_r=x_{r_1}\dotsm x_{r_{d_2}}$, be a simplicial Nagata
  polynomial of (positive) degree $d_1$, where
  $\displaystyle n\leq\binom{m}{d_2}$, let ${\Delta}_f$ be the
  simplicial complex 
  associated to $f$ and let $A=Q/\Ann(f)$. Then
  \begin{equation*}
    A=\bigoplus_{h=0}^{d=d_1+d_2}A_h
  \end{equation*}
  where
  \begin{equation*}
    A_h=A_{(h,0)}\oplus\dotsb\oplus A_{(p,q)}\oplus\dotsb\oplus A_{(0,h)},\,p\leq d_1,\,q\leq d_2,\,A_d=A_{(d_1,d_2)} 
  \end{equation*}
  and moreover,
    $\forall j\in\{0,1,\dotsc,d_2\}$,
  \begin{equation*}
    \dim A_{(i,j)}=a_{i,j}=\begin{cases}
      f_j & i=0\\
      \displaystyle\sum_{r=0}^n f_{j,r} & i\in\{1,\dotsc,d_1-1\},\\  
      f_{d_2-j} & i=d_1
    \end{cases}
  \end{equation*}

  where:
  \begin{itemize}
      \item $f_j$ is the number of $(j-1)$-cells of the 
    ${\Delta}_f$ (with the convention that $f_0=1$);
  \item $f_{j,r}$ is the number of $(j-1)$-subcells of $\Delta_{g_r}$,
    i.e. the $(d_2-1)$-cell 
    of the 
    ${\Delta}_f$ associated to $g_r$ (with the convention that
    $f_{0,r}=1$, so that $\dim A_{(i,0)}=n+1$). 
  \end{itemize}
  More precisely, a basis for $A_{(i,j)}$,
  $\forall j\in\{0,1,\dotsc,d_2\}$, is given by
  \begin{enumerate}
  \item\label{c1:1} If $i=0$,
    $\{\Omega_1,\dotsc,\Omega_{f_j}\}$, 
    where any $\Omega_s:=U_{s_1}\dotsm U_{s_{j}}$ is associated to the
    $(j-1)$-subcell $\{u_{s_1},\dotsc, u_{s_{j}}\}$ of ${\Delta}_f$;
  \item\label{c1:2} If $i=1,\dotsc, d_1-1,\,\left\{\Omega_s^{i,s_1,\dotsc,s_j}\right\}_{\substack{s\in\{0,\dotsc,n\}\\
        s_1,\dotsc,s_j\in\left\{r_1,\dotsc,r_{d_2}\right\}}}$ where
    $\Omega_s^{i,s_1,\dotsc,s_j}:=X_s^iU_{s_1}\dotsm U_{s_j}$ is
    associated to the $(j-1)$-subcell $\{u_{s_1},\dotsc, u_{s_{j}}\}$
    of
    $\Delta_{g_s}(\subset\Delta_f)$; 
  \item\label{c1:3} If $i=d_1$,
    $\left\{X_0^{d_1}\Omega_1(f),\dotsc
      X_n^{d_1}\Omega_{f_{d_2}-j}(f)\right\}$, where
    $\{\Omega_1,\dotsc\Omega_{f_{d_2-j}}\}$ is the basis for
    $A_{(0,d_2-j)}$ of case \eqref{1:1}.
  \end{enumerate}
  In the cases \eqref{c1:1} and \eqref{c1:2} the bases are given by
  monomials, in the case \eqref{c1:3},
    in general, not.
\end{corollary}

\begin{theorem}\label{thm:2}
  Let $\displaystyle f=\sum_{r=0}^n x^{d_1}_r g_r\in S_{(d_1,d_2)}$,
  with $g_r=x_{1}^{r_1}\dotsm x_{m}^{r_m}$ such that
  $r_1+\dotsm+r_m=d_2$, be a 
  CW-Nagata polynomial whose associated CW-complex is $\Delta_f$, as in
  the preceding theorem.

  Then $I:=\Ann(f)$ is generated by:
  \begin{enumerate}
  \item\label{2:1} $X_i X_j$ and $X_k^{d_1+1}$, for
    $i,j,k\in\{0,\dotsc,n\}$, $i<j$;
  \item\label{2:2} $\gen{U_1,\dotsc,U_m}^{d_2+1}$, i.e. all the
    (monic) monomials of degree $d_2+1$;

  \item\label{2:3} The monomials $U_{1}^{s_1}\dotsm U_{m}^{s_m}$ such
    that $s_1+\dotsm+ s_m=j$, where
    $X_{u_{1}^{s_1}\dotsm u_{m}^{s_m}}$ is a (minimal) element of the
    $(j-1)$-skeleton of $P(m)$ not contained in $\Delta_f$ (for
    $j\in\{1,\dotsc,d_2\}$);
  \item\label{2:4} The monomials $X_rU_i$, where
    $u_i$ 
    does not divide $g_r$ (i.e. $\{u_{i}\}$ is not an element of the
    $0$-skeleton 
    of $X_{g_r}$);
      \item\label{2:5} The monomials $X_s U_{1}^{r_1}\dotsm U_{m}^{r_m}$
        such that $r_1+\dotsm+ r_m=j$, where
    $u_{1}^{r_1}\dotsm u_{m}^{r_m}$ is minimal among those that do not
    divide $g_s$ (i.e. the (minimal) element of the $(j-1)$-skeleton
    of $P(m)$, $X_{u_{1}^{r_1}\dotsm u_{m}^{r_m}}$, is not contained
    in $X_{g_s}$),
        for 
    $j\in\{1,\dotsc,d_2\}$;
  \item\label{2:6} The binomials
    $X_r^{d_1}U_1^{\rho_1}\dotsm
    U_m^{\rho_m}-X_s^{d_1}U_1^{\sigma_1}\dotsm U_m^{\sigma_m}$ with
    $\rho_1+\dotsm+\rho_m=\sigma_1+\dotsm+\sigma_m=j$ such that
    $g_{r,s}=\GCD(g_r,g_s)$ and
    $g_r=u_1^{\rho_1}\dotsm u_m^{\rho_m}g_{r,s}$,
    $g_s=u_1^{\sigma_1}\dotsm u_m^{\sigma_m}g_{r,s}$
    (i.e. $X_{g_{r,s}}$ is the element of the $(d_2-j-1)$-skeleton of
    $\Delta_f$ which represents the intersection of $X_{g_r}$ and
    $X_{g_s}$: $X_{g_{r,s}}=X_{g_r}\cap X_{g_s}$).
  \end{enumerate}
\end{theorem}
\begin{proof}
    Let  $A:=T/I$, where
  $T=\KK[X_0,\dotsc,X_n,U_1,\dotsc,U_m]$.
  
  By Theorem \ref{thm:1}, \eqref{1:1} a basis for $A_{(0,j)}$,
  $\forall j\in\{1,\dotsc,d_2\}$, is
  $\{\Omega_1,\dotsc,\Omega_{f_j}\}$, where
  $\Omega_s:=U_{1}^{s_1}\dotsm U_m^{s_{m}}$, with $s_1+\dotsm+s_m=j$,
  is associated to the element $X_{u_{1}^{s_1}\dotsm u_m^{s_{m}}}$ of
  the $(j-1)$-skeleton of $\Delta_f$. 
    Therefore, using the identification introduced in Remark
  \ref{rmk:1},a basis
  for 
  $I_{(0,j)}$ is given by
    the monomials $U_{1}^{s_1}\dotsm U_{m}^{s_m}$ such that
  $s_1+\dotsm+ s_m=j$, where $X_{u_{1}^{s_1}\dotsm u_{m}^{s_m}}$ is an
  element of the $(j-1)$-skeleton of $P(m)$ not contained in
  $\Delta_f$ (for $j\in\{1,\dotsc,d_2\}$);

   Observe that $X_i X_j(f)=0$ if $i\neq j$ and
  $X_k^{d_1+1}(f)=0=U_1^{i_1}\dotsm U_m^{i_m}(f)$ with
  $\sum_{j=1}^mi_j=d_2+1$, for degree reasons.  Set $$\beta:=
  (X_0X_1,\dotsc,X_{n-1}X_n,X_0^{d_1+1},\dotsc,X_n^{d_1+1},\gen{U_1,\dotsc,U_m}^{d_2+1});$$
  this is a homogeneous ideal such that $
  \beta\subset I$ and
  $\displaystyle A\cong\frac{T}{\beta}/\frac{I}{\beta}$.

  By Theorem \ref{thm:1}, \eqref{1:2}, if $i=1,\dotsc, d_1-1$, a basis
  for $A_{(i,j)}$ $\forall j\in\{1,\dotsc,d_2\}$, is given by
  $$\left\{\Omega_s^{i,s_1,\dotsc,s_m}\right\}_{\substack{s\in\{0,\dotsc,n\}\\
      s_k\le,r_k,k=1,\dotsc,m\\ \sum_k s_k=j}}$$ where
  $\Omega_s^{i,s_1,\dotsc,s_m}:=X_s^i\cdot U_{1}^{s_1}\dotsm
  U_m^{s_{m}}$ is associated to the element
  $X_{u_{1}^{s_1}\dotsm u_m^{s_{m}}}$ of the $(j-1)$-skeleton of
  $X_{g_s}$.

   Again using the identification introduced in Remark
  \ref{rmk:1}, a basis for
  $\displaystyle\left(\frac{I}{\beta}\right)_{(i,j)}$ is given by 
  \begin{itemize}
  \item\label{2:4} The monomials $X_r^i U_{1}^{s_1}\dotsm U_{m}^{s_m}$
    such that $s_1+\dotsm+ s_m=j$, with $r\neq s$, where
    $u_{1}^{s_1}\dotsm u_{m}^{s_m}$ divides $g_s$
    (i.e. $X_{u_{1}^{s_1}\dotsm u_{m}^{s_m}}$ is an element of the
    $(j-1)$-skeleton of $X_{g_s}$),
        for $i=1,\dotsc, d_1-1$, and
  \item\label{2:5} The monomials $X_s^i U_{1}^{r_1}\dotsm U_{m}^{r_m}$
    such that $r_1+\dotsm+ r_m=j$, where
    $u_{1}^{r_1}\dotsm u_{m}^{r_m}$ does not divide $g_s$ (i.e. the
    element of the $(j-1)$-skeleton of $P(m)$,
    $X_{u_{1}^{r_1}\dotsm u_{m}^{r_m}}$, is not contained in
    $X_{g_s}$),
      \end{itemize}
  for $j\in\{1,\dotsc,d_2\}$.

  It remains to find the generators of $I$ of bidegree $(d_1,j)$, with
  $j\in\{1,\dotsc,d_2\}$.  This is more complicated since the
  generators of $A_{(d_1,j)}$ are not monomials.  Let $\gamma$ be the
  homogeneous ideal generated by the monomials of the cases
  \eqref{2:1}, \eqref{2:2},
  \eqref{2:3}, \eqref{2:4} and \eqref{2:5}, i.e. the generators that
  we have found so far.  We have $
  \beta\subset\gamma\subset I$ and the exact sequence \eqref{SES}
  given by evaluation at $f$ becomes
    \begin{equation*}
    0\to \left( \frac{I}{\gamma}\right)_{(d_1,j)}\to \left(\frac{T}{\gamma}\right)_{(d_1,j)}\to A_{(0,d_2-j)}\to 0,
  \end{equation*}
  since we identify
  $\displaystyle A\cong\frac{T}{\gamma}/\frac{I}{\gamma}$.  Then, if
  $\rho_1+\dotsm+\rho_m=\sigma_1+\dotsm+\sigma_m=j$,
  $X_r^{d_1}U_1^{\rho_1}\dotsm
  U_m^{\rho_m}-X_s^{d_1}U_1^{\sigma_1}\dotsm
  U_m^{\sigma_m}\in\left(\frac{T}{\gamma}\right)_{(d_1,j)}$ is in
  $\displaystyle\left(\frac{I}{\gamma}\right)_{(d_1,j)}$ if and only
  if
  $X_r^{d_1}U_1^{\rho_1}\dotsm
  U_m^{\rho_m}=X_s^{d_1}U_1^{\sigma_1}\dotsm U_m^{\sigma_m}\in
  A_{(0,d_2-j)}$, which means
  $U_1^{\rho_1}\dotsm U_m^{\rho_m}(g_r)=U_1^{\sigma_1}\dotsm
  U_m^{\sigma_m}(g_s)$.  Since $A_{(0,d_2-j)}$ is generated by the
  monomials $\Omega_s:=U_{1}^{s_1}\dotsm U_m^{s_{m}}$, with
  $s_1+\dotsm+s_m=d_2-j$, associated to the elements of the
  $(d_2-j-1)$-skeleton of $\Delta_f$, we obtain case \eqref{2:6}.

  \end{proof}

As we have done for Theorem \ref{thm:1}, we give, as a corollary of
Theorem \ref{thm:2} the case of the simplicial Nagata polynomial,
giving an improvement of the second part of \cite[Theorem 3.5]{CGIM}; we
also correct that statement, since the authors forgot the 
generators $X_iX_j$, $i\neq j$.

\begin{corollary}\label{cor:2}
  Let $\displaystyle f=\sum_{r=0}^nx^{d_1}_rg_r\in R_{(d_1,d_2)}$,
  with $g_r=x_{r_1}\dotsm x_{r_{d_2}}$, be a simplicial Nagata
  polynomial whose associated 
  simplicial complex is ${\Delta}_f$, as in the preceding theorem.
 
  Then $I:=\Ann(f)$ is generated by:
  \begin{enumerate}
  \item\label{c2:1} $X_iX_j$ and $X_k^{d_1+1}$, for
    $i,j,k\in\{0,\dotsc,n\}$, $i<j$;
  \item\label{c2:2} $U_1^2,\dotsc,U_m^2$;
  \item\label{c2:3} The monomials $U_{s_1}\dotsm U_{s_j}$, where
    $\{u_{s_1},\dotsc, u_{s_j}\}$ is a (minimal) $(j-1)$-cell of
    $2^{\{u_1,\dotsc,u_m\}}$ not contained in ${\Delta}_f$ (for
    $j\in\{1,\dotsc,d_2\}$);
  \item\label{c2:4} The monomials $X_r U_i$, where
    $u_i$ 
    does not divide $g_s$ (i.e. $\{u_i\}\notin \Delta{g_r}$
    ); 
      \item\label{c2:6} The binomials $X_r^{d_1}U_{\rho_1}\dotsm
    U_{\rho_j}-X_s^{d_1}U_{\sigma_1}\dotsm
    U_{\sigma_j}$ such that
    $g_{r,s}\GCD(g_r,g_s)$, $g_r=u_{\rho_1}\dotsm
    u_{\rho_j}g_{r,s}$, $g_s=u_{\sigma_1}\dotsm
    u_{\sigma_j}g_{r,s}$ (i.e.
    $g_{r,s}$ represents the
    $(d_2-j-1)$-face given by the intersection
    $\Delta_{g_r}\cap\Delta_{g_s}$ of the facets of $g_r$ and
    $g_s$: $\Delta_{g_{r,s}}= \Delta_{g_r}\cap\Delta_{g_s}$).

  \end{enumerate}
\end{corollary}

\begin{proof}
  We note only that we have to add the squares of case \eqref{c2:2}
  although they do not correspond to cells, since the polynomials
  $g_i$  are square-free. The rest follows  from Theorem
  \ref{thm:2}. We observe that these squares are in case \eqref{2:2} of Theorem
  \ref{thm:2}.
\end{proof}

\begin{example}
  Let
  \begin{displaymath}
    f=x_0^du_1u_2u_3+x_1^du_1u_2u_4+x_2^du_1u_4u_5+x_3^du_1u_3u_5+x_4^du_2u_3u_6+x_5^du_2u_4u_6+x_6^du_4u_5u_6+x_7^du_3u_5u_6
  \end{displaymath} 
  be a bihomogeneous bidegree $(d,3)$ polynomial with $d\geq1$; it is
  a simplicial Nagata polynomial, whose associated simplicial complex
  is in the following figure:
  
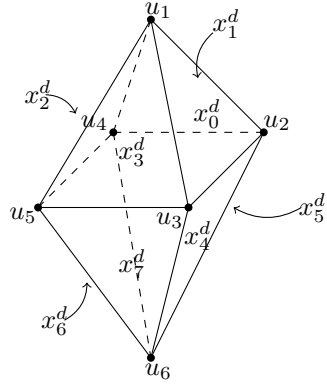
\begin{center}
    \begin{tikzpicture}[scale=0.50]
      \path (0,0) edge[dashed, left] (2,2); \path (2,2) edge[dashed,
      right] (6,2); \path (0,0) edge[ left] (4,0); \path (4,0) edge[
      right] (6,2); \path (4,0) edge[ left] (3,-4); \path (3,-4) edge[
      right] (6,2); \path (4,0) edge[ right] (3,5); \path (3,5) edge[
      left] (6,2); \path (0,0) edge[ left] (3,5); \path (0,0) edge[
      right] (3,-4); \path (2,2) edge[dashed] (3,-4); \path (2,2)
      edge[dashed, right] (3,5); \fill (0,0) circle(3pt); \fill (2,2)
      circle(3pt); \fill (4,0) circle(3pt); \fill (6,2) circle(3pt);
      \fill (3,5) circle(3pt); \fill (3,-4) circle(3pt); \node at
      (3.2,5.25) {$u_1$}; \node at (6.35,2.25){$u_2$}; \node at
      (3.5,-0.35){$u_3$}; \node at (1.5,2.25){$u_4$}; \node at
      (-0.4,-0.2){$u_5$}; \node at (3.2,-4.35){$u_6$}; \node at
      (4.50,2.5){$x^d_0$}; \node at (5,4.9){$x^d_1$}; \node at
      (0,3){$x^d_2$}; \node at (2.5,1.50) {$x^d_3$}; \node at
      (4.25,-0.75) {$x^d_4$}; \node at (7.3,0) {$x^d_5$}; \node at
      (0.45,-3){$x^d_6$}; \node at (2.5,-1.5){$x^d_7$}; \path
      (4.65,4.65) edge[->,bend right] (4.25,3.45); \path (0.2,3)
      edge[->, bend left] (1,2.5); \path (7,0) edge[->,bend left]
      (5.2,0); \path (0.7,-3) edge[->,bend right] (1.3,-2);
    \end{tikzpicture}
  \end{center}
  We have:
  \begin{displaymath}
    A=A_0\oplus A_1\oplus\hdots\oplus A_{d+3}.
  \end{displaymath}
We want firstly to compute the Hilbert vector by applying
  Corollary \ref{cor:1}; first of all,
  \begin{align*}
    a_{1,0}&=8 & a_{0,1}&=6,
  \end{align*}
  and therefore
  \begin{align*}
    h_0&=h_{d+3}=1 \\
    h_1&=h_{d+2}=a_{1,0}+a_{0,1}=8+6=14.
  \end{align*} 
  Then, we analyze the possible cases depending on the degree $d$:
  \begin{itemize}
  \item If $d=1$, then
    \begin{align*}
      a_{1,1}&=8\cdot3=24\\
      a_{0,2}&=12\\
      h_2&=a_{1,1}+a_{0,2}=36
    \end{align*}
    and the Hilbert vector is $(1,14,36,14,1)$.
  \item If $d=2$, then, recalling bigraded Poincar\'e duality,
    \begin{align*}
      a_{2,0}&=a_{0,3}=8 & a_{2,1}&=a_{0,2}=12 
    \end{align*}
    and therefore
    \begin{align*}
      h_2&=a_{2,0}+a_{1,1}+a_{0,2}=8+8\cdot3+12=44,\\
      h_3&=0+a_{2,1}+a_{1,2}+a_{0,3}=8+8\cdot3+8=44
    \end{align*}
    in accordance with Poincar\'e duality; so the Hilbert vector is
    $(1,14,44,44,14,1)$ (cfr.~\cite[Example 3.6]{CGIM}).
  \item If $d=3$, then, again by bigraded Poincar\'e duality,
    \begin{align*}
      a_{3,0}&= a_{0,3}=8, & a_{2,1}&=a_{1,2}=8\cdot3=24, & a_{3,1}&=a_{0,2}=12, &a_{2,2}&=a_{1,1}=24, &a_{1,3}&=a_{2,0}=8, 
    \end{align*}
    therefore
    \begin{align*}
      h_2&=a_{2,0}+a_{1,1}+a_{0,2}=44,\\
      h_3&=a_{3,0}+a_{2,1}+a_{1,2}+a_{0,3}=64\\
      h_4&=0+a_{3,1}+a_{2,2}+a_{1,3}=44
    \end{align*}
    $h_2=h_4$ in accordance with Poincar\'e duality and the Hilbert
    vector is $(1,14,44,64,44,14,1)$.
  \item In general, let $d\geq4$; by hypothesis
    \begin{equation*}
      h_{d+1}=h_2=a_{2,0}+a_{1,1}+a_{0,2}=44,
    \end{equation*}
    and
    \begin{align*}
      h_k&=a_{k,0}+a_{k-1,1}+a_{k-2,2}+a_{k-3,3} & \forall k&\in\{3,\dotsc,d\}, 
    \end{align*}
    where
    \begin{align*}
      a_{k,0}&=8 & a_{k-1,1}&=8\cdot3=24 & a_{k-2,2}&=8\cdot3=24 & a_{k,3}&=8. 
    \end{align*}
    Again using the Poincar\'e duality we have:
    \begin{align*}
      h_{d+3-k}&=h_k=64 & \forall k&\in\left\{3,\hdots,\left\lfloor\frac{d+3}{2}\right\rfloor\right\}
    \end{align*}
    and the Hilbert vector is $(1,14,44,64,\dotsc,64,44,14,1)$.
  \end{itemize}

  Now, we want to find the generators of $\Ann(f)$, by applying
  Corollary \ref{cor:2}. Behaviour depends on  $d$:

  \begin{itemize}
  \item If  $d=1$, by Corollary \ref{cor:2} $\Ann(f)$ is (minimally)
    generated by:
    \begin{enumerate}
    \item $\langle X_0,\hdots,X_7\rangle^2=X_0^2,X_0X_1,\dotsc;$
    \item $U_1^2,\hdots,U_6^2$;
    \item $U_1U_6,U_2U_5,U_3U_4$;
    \item $X_0U_4,X_0U_5,X_0U_6,X_1U_3,X_1U_5,X_1U_6,X_2U_2,X_2U_3,X_2U_6,X_3U_2,X_3U_4,X_3U_6,\\
      X_4U_1,X_4U_4,
      X_4U_5,X_5U_1,X_5U_3,X_5U_5,X_6U_1,X_6U_2,X_6U_3,X_7U_1,X_7U_2,X_7U_4$;
    \item $X_0U_3-X_1U_4,X_0U_2-X_3U_5,X_0U_1-X_4U_6,X_1U_2-X_2U_5,X_1U_1-X_5U_6,\\
      X_2U_4-X_3U_3,X_2U_1-X_6U_6,X_3U_1-X_7U_6,X_4U_3-X_5U_4,X_4U_2-X_7U_5,\\
      X_5U_2-X_6U_5,X_6U_4-X_7U_3$.
    \end{enumerate}

  \item If $d\geq2$, by Corollary \ref{cor:2} $\Ann(f)$ is
    (minimally) generated by
    \begin{enumerate}
    \item $\langle X_0,\hdots,X_7\rangle^{d+1}$ and $X_hX_k$ where
      $h,k\in\{0,\hdots,7\}$, $h<k$;
    \item $U_1^2,\hdots,U_6^2$;
    \item $U_1U_6,U_2U_5,U_3U_4$;
    \item $X^d_0U_4,X^d_0U_5,X^d_0U_6,X^d_1U_3,X^d_1U_5,X^d_1U_6,X^d_2U_2,X^d_2U_3,X^d_2U_6,X^d_3U_2,X^d_3U_4,X^d_3U_6,\\
      X^d_4U_1,X^d_4U_4,X^d_4U_5,X^d_5U_1,X^d_5U_3,X^d_5U_5,X^d_6U_1,X^d_6U_2,X^d_6U_3,X^d_7U_1,X^d_7U_2,X^d_7U_4$;
    \item $X^d_0U_3-X^d_1U_4,X^d_0U_2-X^d_3U_5,X^d_0U_1-X^d_4U_6,X^d_1U_2-X^d_2U_5,X^d_1U_1-X^d_5U_6,\\
      X^d_2U_4-X^d_3U_3,X^d_2U_1-X^d_6U_6,X^d_3U_1-X^d_7U_6,X^d_4U_3-X^d_5U_4,X^d_4U_2-X^d_7U_5,\\
      X^d_5U_2-X^d_6U_5,X^d_6U_4-X^d_7U_3$.
    \end{enumerate}
  \end{itemize}
\end{example}

\begin{example}
  Let
  \begin{displaymath}
    f=x_0^du_1u_2+x_1^du_1^2+x^d_2u_2u_3
  \end{displaymath}
  be a bihomogeneous bidegree $(d,2)$ polynomial, with $d\geq1$; it is
  a CW-Nagata polynomial whose CW-complex is the following:
  
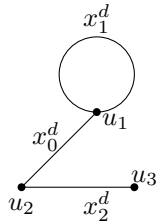
\begin{center}
    \begin{tikzpicture}[scale=0.5]
      \path (0,0) edge[left] (1,1); \path (1,1) edge[right] (2,2);
      \path (0,0) edge[right] (3,0); \node at (0.65,1.35) {$x_0^d$};
      \draw (2,3) circle (1cm); \node at (2,4.5) {$x_1^d$}; \fill
      (2,2) circle(3pt); \node at (2.5,1.75) {$u_1$}; \fill (0,0)
      circle(3pt); \node at (0,-0.5) {$u_2$}; \fill (3,0) circle(3pt);
      \node at (3.275,0.25) {$u_3$}; \node at (2,-0.5) {$x_2^d$};
    \end{tikzpicture}
  \end{center}
  We have:
  \begin{displaymath}
    A=A_0\oplus A_1\oplus\hdots\oplus A_{d+2}
  \end{displaymath}
  and we want to find its Hilbert vector; first of all,
  \begin{align*}
    a_{1,0}&=3 & a_{0,1}&=3
  \end{align*}
  and therefore
  \begin{align*}
    h_0&=h_{d+2}=1 & h_1&=h_{d+1}=a_{1,0}+a_{0,1}=6. 
  \end{align*}
  Therefore, if $d=1$, then Hilbert vector is $(1,6,6,1)$.

  If $d=2$, we have
  \begin{equation*}
    a_{1,1}=2+1+2=5,
  \end{equation*}
  so
  \begin{equation*}
    h_2=\dim A_2=a_{2,0}+a_{1,1}+a_{0,2}=3+5+3=11
  \end{equation*}
  and the Hilbert vector is $(1,6,11,6,1)$.

  If $d=3$ then, by bigraded  Poincar\'e duality
  \begin{align*}
    a_{3,0}&= a_{0,2}=3 & a_{0,3}&=3
  \end{align*}
  so
  \begin{align*}
    h_2&=a_{2,0}+a_{1,1}+a_{0,2}=11\\
    h_3&=a_{3,0}+a_{2,1}+a_{1,2}+a_{0,3}=3+5+3=11
  \end{align*}
  and the Hilbert vector is $(1,6,11,11,6,1)$.

  In general, let $d\geq4$; by hypothesis
  \begin{equation*}
    h_d=h_2=a_{2,0}+a_{1,1}+a_{0,2}=11,
  \end{equation*}
  and
  \begin{align*}
    h_k&=\dim A_{(k,0)}+\dim A_{(k-1,1)}+\dim A_{(k-2,2)}& \forall k&\in\{3,\dotsc,d\},
  \end{align*}
  so, since
  \begin{align*}
    a_{k,0}&=3 & a_{k-1,1}&=5 & a_{k-2,2}&=3
  \end{align*}
  using Poincar\'e duality we have:
  \begin{align*}
    h_{d+2-k}&=h_k=a_{k,0}+a_{k-1,1}+a_{k-2,2}=11 & \forall k&\in\left\{3,\hdots,\left\lfloor\frac{d+2}{2}\right\rfloor\right\}, 
  \end{align*}
  and the Hilbert vector is $(1,6,11,\hdots,11,6,1)$. 

 Let $d=1$, by Theorem \ref{thm:2} $\Ann(f)$ is (minimally)
  generated by:
  \begin{itemize}
  \item $\langle X_0,X_1,X_2\rangle^2,U_2^2,U_3^2,U_1U_3,U_1^3$;
  \item $X_0U_1^2,X_0U_3,X_1U_2,X_1U_3,X_2U_1$;
  \item $X_0U_2-X_1U_1,X_0U_1-X_3U_3$.
  \end{itemize}
  Let $d\geq2$, by Theorem \ref{thm:2} $\Ann(f)$ is
  (minimally) generated by:
  \begin{itemize}
  \item
    $\langle
    X_0,X_1,X_2\rangle^{d+1},X_0X_1,X_0X_2,X_1X_2,U_2^2,U_3^2,U_1U_3,U_1^3$;
  \item $X^d_0U_1^2,X^d_0U_3,X^d_1U_2,X^d_1U_3,X^d_2U_1$;
  \item $X^d_0U_2-X^d_1U_1,X^d_0U_1-X^d_3U_3$.
  \end{itemize}
\end{example}

\end{document}